\documentclass[11pt]{amsart}

\usepackage{amsmath,amssymb,amsfonts}

\usepackage[all]{xy}

\usepackage[T5,T1]{fontenc}
\usepackage{mathpazo}

\usepackage{hyperref}
\usepackage{a4wide}

\theoremstyle{plain}
\newtheorem{thm}{Theorem}[section]
\newtheorem{lem}[thm]{Lemma}
\newtheorem{cor}[thm]{Corollary}
\newtheorem{prop}[thm]{Proposition}

\newtheorem{conj}[thm]{Conjecture}

\theoremstyle{definition}
\newtheorem{defn}[thm]{Definition}
\newtheorem{ex}[thm]{Example}

\newtheorem{rmk}[thm]{Remark}

\renewcommand\appendix{
\section*{Appendix}
\gdef\thesection{\Alph{section}}
\setcounter{section}{1}
\newtheorem{appxex}{Example}[section]

\newtheorem{appxprop}[appxex]{Proposition}
\newtheorem{appxrmks}[appxex]{Remarks}
}

\newcommand{\sC}{{\mathcal C}}

\newcommand{\sL}{{\mathcal L}}

\newcommand{\C}{{\mathbb C}}

\newcommand{\F}{{\mathbb F}}

\newcommand{\N}{{\mathbb N}}

\newcommand{\Q}{{\mathbb Q}}
\newcommand{\R}{{\mathbb R}}

\newcommand{\U}{{\mathbb U}}

\newcommand{\Z}{{\mathbb Z}}

\def\NDT{{\fontencoding{T5}\selectfont Nguy\~ \ecircumflex n Duy T\^an}}

\begin{document}
\title[Kernel Unipotent Conjecture and Massey Products over rigid fields]{The kernel unipotent conjecture and the Vanishing of  Massey products  for odd rigid fields}
 \author[J. Min\'a\v{c} and N. D. T\^an (with an appendix by I. Efrat, J. Min\'a\v{c} and N. D. T\^an)]{ J\'an Min\'a\v{c} and \NDT\\
 (with an appendix by Ido Efrat, J\'an Min\'a\v{c} and \NDT)}
  \address{Department of Mathematics, Ben Gurion University of the Negev, P.O.  Box 653, Be'er-Sheva 84105
Israel}
\email{efrat@math.bgu.ac.il }
\address{Department of Mathematics, Western University, London, Ontario, Canada N6A 5B7}
\email{minac@uwo.ca}
 \address{Department of Mathematics, Western University, London, Ontario, Canada N6A 5B7 and Institute of Mathematics, Vietnam Academy of Science and Technology, 18 Hoang Quoc Viet, 10307, Hanoi - Vietnam } 
\email{dnguy25@uwo.ca}
\subjclass[2010]{Primary 12F10, 20F14; Secondary 12G05}
\thanks{IE was supported by the Israel Science Foundation (grant No. 152/13). JM is partially supported  by the Natural Sciences and Engineering Research Council of Canada (NSERC) grant R0370A01. NDT is partially supported  by the National Foundation for Science and Technology Development (NAFOSTED)}
 \begin{abstract}
 A major difficult problem in Galois theory is the characterization of profinite groups which are realizable as absolute Galois groups of fields. Recently the Kernel $n$-Unipotent Conjecture and the Vanishing $n$-Massey Conjecture for $n\geq 3$ were formulated. These conjectures evolved in the last forty years as a byproduct of the application of topological methods to Galois cohomology. We show that both of these conjectures are true for odd rigid fields. This is the first case of a significant family of fields where both of the conjectures are verified besides fields whose Galois groups of $p$-maximal extensions are free pro-$p$-groups.
We also prove the Kernel Unipotent Conjecture for Demushkin groups of rank 2, and establish various filtration results for free pro-$p$-groups, provide examples of pro-$p$-groups which do not have the kernel $n$-unipotent property, compare various Zassenhaus filtrations with the descending $p$-central series and establish new type  of automatic Galois realization.
\end{abstract}
\maketitle
\section{Introduction}
In 1928, W. Krull in \cite{Kr} introduced a topology for a general Galois group. Thus each Galois group is a profinite topological group. If $F$ is a field and $F_{\rm sep}$ is its separable closure, then $G_F={\rm Gal}(F_{\rm sep}/F)$ - the Galois group of $F_{\rm sep}$ over $F$ - is  called the absolute Galois group of $F$. What special properties do  absolute Galois groups have among all profinite groups?
In the classical papers \cite{AS1,AS2} published in 1927, E. Artin and O. Schreier developed a theory of real fields, and they showed in particular that the only non-trivial finite subgroups of absolute Galois groups are groups of order 2. In \cite{Be}, E. Becker developed some parts of Artin-Schreier theory by replacing separable closures of fields by maximal $p$-extensions of fields.
Conjectures~\ref{conj:vanishing n-Massey} and \ref{conj:kernel intersection} below describe conjecturally rather important special properties of absolute Galois groups and their maximal pro-$p$-quotients. In \cite{MT1} we explained how these conjectures evolved over a number of years.

Let $G$ be a profinite group and let $\F_p$ be a field with $p$ elements considered as a discrete $G$-module with trivial action. For each $i\in \N\cup\{0\}$, let $H^i(G,\F_p)$ be the $i$-th cohomology of $G$  with $\F_p$-coefficients. Let $\alpha_1,\ldots,\alpha_n\in H^1(G,\F_p)$. In Section~\ref{sec:Review Massey}, we recall the definition of an $n$-fold Massey product
\[
 \langle \alpha_1,\ldots,\alpha_n\rangle \subseteq H^2(G,\F_p),
\]
when it is defined. So an $n$-fold Massey product is not always defined, and when it is defined, in general it is not a single-valued function of $\alpha_1,\ldots,\alpha_n$ but it is a subset of $H^2(G,\F_p)$. Nevertheless the significance of $n$-Massey products for Galois theory is considerable. An important theorem is Dwyer's theorem \cite[Theorem 2.4]{Dwy75}, quoted as Theorem~\ref{thm:Dwyer} below in our paper. It shows that certain lifts of unipotent representations are possible if and only if $\langle \alpha_1,\ldots,\alpha_n\rangle$ is defined and $0\in\langle \alpha_1,\ldots,\alpha_n \rangle$.  If this is so, then we say that $\langle \alpha_1,\ldots,\alpha_n\rangle$ vanishes. Assume that a prime $p$ and an integer $n\in \N$ are given. If, for each $\langle \alpha_1,\ldots,\alpha_n \rangle$ which is defined, we have that $\langle \alpha_1,\ldots,\alpha_n\rangle$ vanishes, then we say that $G$ has the vanishing $n$-fold Massey product property with respect to $\F_p$.  

In searching for other significant Galois groups which could play an analogous role as dihedral groups of oder 8, play in the quotients of $G_F(2)$ by its third descending 2-central series, the authors of \cite{GLMS}  found certain groups $G_1$ and $G_2$ of order $32$ and $64$ respectively. In 2004 J. Labute, during his visist to the first author, pointed out the significance of Massey products for Galois theory, and he pointed out two remarkable works \cite{Mo1} and \cite{Vo}. In 2011 the first author, while working with I. Efrat, observed the relevance of Massey products with work in \cite{GLMS}. Then in 2013, I. Efrat \cite{Ef} investigated a property of profinite groups, what we later christened as the kernel $n$-unipotent property bellow. In \cite{MT1} it was observed that $G_2\simeq \U_4(\F_2)$ and $G_1$ is isomorphic to a subgroup of $\U_4(\F_2)$. Moreover it was shown that $G_1$ and $G_2$ play an important role in vanishing triple Massey products with respect to $p=2$. 
Finally the following conjecture was formulated in \cite{MT1}.
\begin{conj}
\label{conj:vanishing n-Massey}
 Let $p$ be a prime number and $n\geq 3$ an integer. Let $F$ be a field, which contains a primitive $p$-th root of unity if ${\rm char}(F)\not=p$. Then the absolute Galois group $G_F$ of $F$ has the vanishing $n$-fold Massey product property with respect to $\F_p$.
\end{conj}

This conjecture is a significant conjecture in Galois theory, and in each case when it can be established for any particular family of fields, any $n\in \N$, $n\geq 3$ and any prime $p$, it has remarkable consequences for the automatic realization of Galois groups and the structure of absolute Galois groups. For further motivations, results and applications, see \cite{MT1}.

In \cite{MT1} we show that Conjecture~\ref{conj:vanishing n-Massey} is true for $n=3,p=2$, and all fields. This is the first case when the validity of this conjecture is established for all fields. This results extends a previous result of M. J. Hopkins and K. G. Wickelgren in \cite{HW}, where they prove this result for global fields of characteristic not 2. In \cite{MT1} we also prove Conjecture~\ref{conj:vanishing n-Massey} for all local fields and all $n\geq 3$, and all primes  $p$. In \cite{MT2} we prove Conjecture~\ref{conj:vanishing n-Massey} for algebraic number fields $n=3$ and all primes $p$.

In \cite{MT1} the Kernel $n$-Unipotent Conjecture recalled below, was also formulated and discussed. In Section~\ref{sec:kernel conj} we recall the definition of the $p$-Zassenhaus filtration $G_{(n)}$ for any profinite group $G$ and prime number $p$. Recall that for a unital commutative ring $\Lambda$, $\U_n(\Lambda)$ is the group of all upper-triangular unipotent $n\times n$-matrices
with entries in  $\Lambda$. The following very interesting property of profinite groups was first studied in \cite{Ef}.

 \begin{defn}
 Let $G$ be a pro-$p$-group and let $n\geq 1$ be an integer. We say that $G$ has the {\it kernel $n$-unipotent property} if 
\[
G_{(n)}=\bigcap \ker(\rho\colon G\to \U_n(\F_p)),
\]
where $\rho$ runs over the set of all representations (continuous homomorphisms) $G\to \U_n(\F_p)$.
\end{defn}

It is easy to see that for $n=1$ and $2$, every pro-$p$-group $G$ has the kernel $n$-unipotent property. In Appendix written jointly with Ido Efrat, we show that for each $n\geq 3$ there exists a finitely generated pro-$p$-group $G$ such that $G$ does not have the kernel $n$-unipotent property. 
In analogy with transcendental numbers, although we assume that many  pro-$p$-groups do not have the kernel $n$-unipotent property, one has to find a suitable family in order to able to check this. 
In a subsequent paper \cite{MT3}, we show that every pro-$p$ Demushkin group has the kernel $n$-unipotent property for $n=3,4$. In Section  5, we also show that pro-$p$ Demushkin groups of rank 2 have the kernel $n$-unipotent property for all $n\geq 3$ (see Proposition~\ref{prop:Demushkin rank 2}).  It is shown in \cite[Theorem A]{Ef} that every free pro-$p$-group has the kernel $n$-property for all $n\geq 3$. This theorem was in \cite{Ef} deduced from a more general theorem called Theorem A' in \cite{Ef}.
 (In Section 2, we provide an alternative direct short proof  (see Theorem~\ref{thm:intersection} part a). That such a proof is possible was announced earlier in \cite[Section 8]{MT1}. Recently, Efrat in \cite{Ef2} also obtained such a proof independently from us.)

It was shown that for $G=G_F(p)$, where $F$ is a field containing a primitive $p$-th root of unity, $G$ has the kernel $3$-unipotent property. (See \cite{MS2,Vi,EM1} for the case $p=2$ and \cite[Example 9.5 (1)]{EM2} for the case $p>2$.) These are deep results closely related to the well-known Merkurjev-Suslin theorem (\cite{MeSu}).

\begin{conj}[Kernel $n$-Unipotent Conjecture]
\label{conj:kernel intersection}
Let $F$ be a field containing  a primitive $p$-th root of unity and let $G=G_F(p)$. Let $n\geq 3$ be an integer. Then $G$ has the kernel $n$-unipotent property.
\end{conj}
A connection between Conjectures~\ref{conj:vanishing n-Massey} and \ref{conj:kernel intersection} is via Dwyer's theorem quoted as Theorem~\ref{thm:Dwyer} below. Namely, in order to obtain $n$-unipotent representations of $G=G_F$ we need $n$-Massey products vanishing for suitable $\alpha_1,\ldots,\alpha_{n-1}\in H^1(G,\F_p)$. 

The main results in our paper are Theorems~\ref{thm:main} and~\ref{thm:p>2}. In them we establish both the Vanishing $n$-Massey Conjecture and the Kernel $n$-Unipotent Conjecture for $p$-rigid fields ($p>2$) and for all $n\geq 3$. We shall now recall the definition of rigid fields.

Assume $p>2$. This assumption is made throughout the paper except when we explicitly consider $p=2$. Let $F$ be a field, which we assume to contain a fixed primitive $p$-th root of unity $\zeta_p$.
For each $a\in F^\times=F\setminus\{0\}$, we have an element $\chi_a\in {\rm Hom}(G_F,\F_p)=H^1(G_F,\F_p)$ defined by $ \sigma (\sqrt[p]{a})= \zeta_p^{\chi_a(\sigma)} \sqrt[p]{a},$ for all $\sigma\in G_F$. 

Then $F$ is $p$-rigid if and only if $\chi_a\cup \chi_b=0\in H^2(G_F,\F_p)$ implies that $\chi_a,\chi_b\in H^1(G,\F_p)$ are linearly dependent over $\F_p$. (This definition  coincides with the one given in Definition~\ref{defn:p-rigid}.)  Since we assume that $p> 2$, we sometimes call such a field an odd rigid field. 
Besides fields $F$ with $G_F(p)$ free pro-$p$-groups or Demushkin pro-$p$-groups, $p$-rigid fields play a fundamental role in current studies of Galois groups of maximal $p$-extensions. (See e.g., \cite{CMQ}, \cite{EK}, \cite{Wa1},\cite{Wa2} and \cite{Wa3}.) 

The structure of this paper is as follows. In Section 2 we will present another (short and direct)  proof for a result, which was first proved by I. Efrat (\cite{Ef}), that  every  free pro-$p$-group has the kernel $n$-unipotent property for all $n$ (Theorem~\ref{thm:intersection} part a). At the same time, we obtain  analogous new results for other filtrations, such as the descending central series and the descending $p$-central series. The result for the descending $p$-central series is interesting because it provides us with first steps toward an analogy to the Kernel $n$-Unipotent Conjecture when we replace the $p$-Zassenhaus filtration by the descending $p$-central series. We also obtain discrete versions of these results (see Theorem~\ref{thm:discrete free}). 
In Section 3 we provide some results on unipotent matrices which we shall later use for constructing unipotent representations. 
In Section 4 we recall some basic facts on  the structure of Galois groups of maximal $p$-extensions  of  odd $p$-rigid fields. Using results obtained in Section 3 and the results recalled in Section 4, we prove the Kernel $n$-Unipotent Conjecture for odd rigid fields in Section 5. 
In this section we also show that every Demushkin group with rank at most $2$ has the kernel $n$-unipotent property for all $n$. In Section 6, for a given $n\geq 3$, we provide an example of a torsion-free pro-$p$-group which does not have the kernel $n$-unipotent property. We also compare the $p$-th and $(p+1)$-th terms in the $p$-Zassenhaus filtration with the third term in the descending $p$-central series for $G_F(p)$, here $F$ is a field containing a primitive $p$-th root of unity (see Proposition~\ref{prop:smallest}).
In Section 7, we review some basic facts on Massey products.  In Section 8 we prove the Vanishing $n$-Massey Conjecture for any odd rigid field and for any $n$. See Theorem~\ref{thm:Massey rigid}. This result is a consequence of a more general result, Theorem~\ref{thm:p>2}. The latter theorem has its own interest because it deals with general fields which contain a primitive $p$-th root of unity if their characteristic is different from $p$, and it also provides an explanation of the well-known result of Artin-Schreier result mentioned at the beginning of the Introduction. In this last section, we also derive a new type of automatic Galois realization (see Corollary~\ref{cor:automatic realization}).  In Appendix written jointly with Ido Efrat, for each $n\geq 3$ we provide examples of pro-$p$-groups which do not satify the kernel $n$-unipotent property.
\\
\\
{\bf Acknowledgements: } We are very grateful to Ido Efrat for his previous collaboration and his important contributions related to the formulation of the Kernel Unipotent Conjecture and results on its validity and to both Ido Efrat and Julien Blondeau for interesting discussions. We would like to thank Sunil Chebolu and  Claudio Quadrelli  for working with the first author on rigid fields.  We would like to thank the referee for his/her careful reading  and  suggestions which we used to improved our exposition.
\section{Kernel conjecture for free-$p$-groups}
\label{sec:kernel conj}

Recall that for a profinite group $G$ and a prime number $p$, the descending central series $(G_i)$, the descending $p$-central series $(G^{(i)})$ and the $p$-Zassenhaus filtration $(G_{(i)})$ of $G$ are defined inductively by
\[
G_1=G,\quad G_{i+1}=[G_i,G],
\]
by
\[
G^{(1)}=G,\quad G^{(i+1)}=(G^{(i)})^p[G^{(i)},G],
\]
and by
\[
G_{(1)}=G, \quad G_{(n)}=G_{(\lceil n/p\rceil)}^p\prod_{i+j=n}[G_{(i)},G_{(j)}],
\]
where $\lceil n/p \rceil$ is the least integer which is greater than or equal to $n/p$. (Here for closed subgroups $H$ and $K$ of $G$, the symbol $[H,K]$ means the smallest closed subgroup of $G$ containing the  commutators $[x,y]=x^{-1}y^{-1}xy$, $x\in H, y\in K$. Similarly, $H^p$ means the smallest closed subgroup of $G$ containing  the $p$-th powers $x^p$, $x\in H$.)

Let $\Lambda$ be $\Z_p$ or $\Z_p/p^r\Z_p=\Z/p^r\Z$ with $r\in \N$. Let  $S$ be a free pro-$p$-group on a finite set of  generators $x_1,\ldots,x_d$.  Then we have the  Magnus homomorphism from the   completed group algebra $\Lambda[[S]]$  to the ring $\Lambda\langle\langle X_1,\ldots,X_d\rangle\rangle$ of the formal power series in $d$ non-commuting variables $X_1,\ldots,X_d$ over $\Lambda$ (equipped with the topology of coefficient-wise convergence).
\[
\psi\colon \Lambda[[S]] \to \Lambda \langle\langle X_1,\ldots,X_d\rangle\rangle, x_i\mapsto 1+X_i.
\]
One basic result is the following
\begin{lem} 
\label{lem:0a}
The Magnus homomorphism  $\psi$ is a (continuous) isomorphism.
\end{lem}
\begin{proof} See, for example, \cite[Chapter I, Proposition 7]{Se} or \cite[Chapter 6]{Laz}.
\end{proof}
 A multi-index $I=(i_1,\ldots,i_k)$ is called {\it of height} $d$ if $1\leq i_r\leq d$. Its {\it length} $|I|$ is $k$. If multi-indices $I=(i_1,\ldots,i_k)$, $J=(j_1,\ldots,j_l)$ are given, we denote by
\[
IJ=(i_1,\ldots,i_k,j_1,\ldots,j_l), 
\]
the concatenation of $I$ and $J$. Also for a multi-index $I=(i_1,\ldots,i_k)$ of height $d$, we set $X_I=X_{i_1}\cdots X_{i_k}$.

The {\it Magnus expansion (relative to $\Lambda$)} $\psi(s)$ of $s\in S$ is given by
\[
 \psi(s)= 1+ \sum_{I} \epsilon_{I,\Lambda}(s) X_I,
\]
where $I$ runs over all multi-indices of height $d$. So for each $I$, we have a map $\epsilon_{I,\Lambda}\colon S\to \Lambda$. For $I=\emptyset$, we define $\epsilon_{\emptyset,\Lambda}(s)=1$, for all $s\in S$.

The next lemma records some quite well-known facts. 
\begin{lem}
\label{lem:1a}
Let $\sigma,\tau$ be elements in $S$.  Let $n\geq 2$ be an integer.
\begin{enumerate}
\item[(a)] For any multi-index $I$ of height $n$,  one has 
\[ \epsilon_{I,\Lambda}(\sigma\tau)=\sum_{I_1I_2=I} \epsilon_{I_1,\Lambda}(\sigma)\epsilon_{I_2,\Lambda}(\tau).\]
\item[(b)] We have $\sigma\in S_{n}$ if and only if $\epsilon_{I,\Z_p}(\sigma)=0$ for every multi-index $I$ with $1\leq |I|<n$.
\item[(c)] We have $\sigma\in S_{(n)}$ if and only if $\epsilon_{I,\F_p}(\sigma)=0$ for every multi-index $I$ with $1\leq |I|<n$.
\item[(d)] We have $\sigma\in S^{(n)}$ if and only if  $v(\epsilon_{I,\Z_p}(\sigma))\geq n-|I|$ for every multi-index $I$ with $1\leq |I|<n$. Here $v$ is the $p$-adic valuation on $\Z_p$ ($v(p)=1$).
\end{enumerate}
\end{lem}
\begin{proof} (a) It just follows from the fact that $\psi$ is a homomorphism.  

(b) See e.g., \cite[Proposition 8.15]{Mo}.

(c) See e.g.,  \cite[Lemma 2.19]{Vo2} and/or \cite[Proposition 6.2]{Ef}. The key point here is that by the Jennings-Brauer theorem (see \cite[Theorem 12.9]{DDMS}), the $p$-Zassenhaus filtration on $S$ coincides with the filtration induced from the $I_{\F_p}(S)$-adic filtration on $\F_p[[S]]$ (here $I_{\F_p}(S)$ is the augmentation ideal of $\F_p[[S]]$). Namely, we have $S_{(n)}=\{s\in S\mid s-1\in I_{\F_p}^{n}(S)\}.$

(d) See e.g., \cite[Theorem 7.14]{Ko} together with Lemma~\ref{lem:0a}, or \cite[Satz 1]{Ko1}. 
\end{proof}
The following lemma was proved in \cite[Lemma 7.4]{Ef}.
\begin{lem}
\label{lem:2a}
Let $I=(i_1,\ldots,i_k)$ be a multi-index. We define a map  
\[ \rho=\rho_{I,\Lambda}\colon  S\to \U_{k+1}(\Lambda),\] 
by
\[\rho(\sigma)_{\mu\nu}=\epsilon_{(i_\mu,\ldots,i_{\nu-1}),\Lambda}(\sigma),
\]
for $\sigma\in S$ and $\mu<\nu$ (the other entries being obvious). Then $\rho$ is a continuous group homomorphism.
\end{lem}
\begin{proof} It follows from Lemma~\ref{lem:1a} a). 
\end{proof}
For $1\leq i,j\leq n$, let $e_{ij}$ be the $n\times n$ matrix with the 1 in $\Lambda$ in the position $(i,j)$ and 0 elsewhere. It is well-known that the set of elements $1+e_{i,i+1}$, $i=1,\ldots, n$, generate the group $\U_n(\Lambda)$. (See \cite[page 455]{We}.) 

\begin{lem} Let $n,r,s\geq 1$ be integers. Then we have 
\[ \U_n(\Z_p)_{n}=\U_n(\F_p)_{(n)}=\U_r(\Z/p^s\Z)^{(r+s-1)}=1.\]
\end{lem}
\begin{proof} We only give a proof for the statement $\U_r(\Z/p^s\Z)^{(r+s-1)}=1$. Other statements are proved similarly and even more easily.

Let $S$ be a free pro-$p$-group on $r-1$ generators $x_1,\ldots,x_{r-1}$. We consider the multi-index $I= (1,\ldots,r-1)$ of length $r-1$. Then by Lemma~\ref{lem:2a}, we have a continuous homomorphism 
\[ \rho=\rho_{I,\Z/p^s\Z}\colon  S\to \U_r(\Z/p^s\Z)\] 
which is defined by
\[\rho(\sigma)_{\mu\nu}=\epsilon_{(\mu,\ldots,\nu-1),\Z/p^s\Z}(\sigma),
\]
for $\sigma\in S$ and $\mu<\nu$ (the other entries being obvious).
In particular $\rho(x_i)=1+e_{i,i+1}$, for $i=1,\ldots,r-1$. Thus $\rho$ is surjective, and hence $\rho(S^{(r+s-1)})=\U_r(\Z/p^s\Z)^{(r+s-1)}$.

Now let $\sigma$ be any element in $S^{(r+s-1)}$. For each sub-multi-index $J\subseteq I$,  we have $v_p(\epsilon_{J,\Z_p}(\sigma))+|J| \geq r+s-1$ by Lemma~\ref{lem:1a} part d). Thus $v_p(\epsilon_{J,\Z_p}(\sigma))\geq s$, and hence $\epsilon_{J,\Z/p^s\Z}(\sigma)=0$. It implies that $\rho(\sigma)=1$. Therefore $\U_r(\Z/p^s\Z)^{(r+s-1)}=\rho(S^{(r+s-1)})=1$.
\end{proof}
\begin{lem} Let $G$ be a pro-$p$-group. 
 \label{lem:3a} 
\begin{enumerate} 
\item[(a)] Every continuous homomorphism $\rho\colon G\to \U_{n}(\Z_p)$ is trivial on $G_{n}$. 
\item[(b)] Every continuous homomorphism $\rho\colon G\to \U_{n}(\F_p)$ is trivial on $G_{(n)}$.
\item[(c)] For each $k=1,\ldots,n-1$, every continuous homomorphism $\rho\colon G\to \U_{k+1}(\Z/p^{n-k}\Z)$ is trivial on $G^{(n)}$. 
\end{enumerate}
\end{lem}
\begin{proof} (a) This follows from the fact that $\U_n(\Z_p)_n=1$.

(b) This follows from the fact that $\U_n(\F_p)_{(n)}=1$.

(c) This follows from the fact that $\U_k(\Z/p^{n+1-k}\Z)^{(n)}=1$.
\end{proof}

Part a) of the following theorem was first proved by I. Efrat in \cite[Theorem A]{Ef}.   (In \cite{MT1} we mentioned that we found a short direct proof of it but we did not write details of this proof.)
After completion of the previous version of our paper, we received from I. Efrat, the preprint \cite{Ef2}, where he also found a direct proof of this fact, and he also characterized  independently from us the $n$-th term $S^{(n)}$ of the descending $p$-central series of a free pro-$p$-group $S$   under a condition that $n<p$ (see \cite[Version 1, Theorem in Introduction]{Ef2}). In our correspondence with Efrat, we clarified that our use of Koch's Theorem 7.14 in \cite{Ko} was correct. In the new version, Efrat was able to remove the extra hypothesis $n<p$. (See \cite[Version 2, Theorem in Introduction]{Ef2}). 
 As we mentioned in the introduction, this characterization of $S^{(n)}$ is important for a possible formulation of an analogue of the Kernel $n$-Unipotent Conjecture for the descending $p$-central series rather than the $p$-Zassenhaus filtration of absolute Galois groups.
 
 In addition in \cite{Ef2} the discrete variant of the version of Theorem~\ref{thm:intersection} is our Theorem~\ref{thm:discrete free} below is also obtained.
  Efrat also points out in \cite[Introduction]{Ef2} that in Theorem~\ref{thm:discrete free} (b) we recover the result of Gr\"{u}n \cite{Gr}. (See also \cite[Example 6.2]{Ef2}.) We refer to the very interesting history of this result to \cite[Chapter 2, Section 7]{CM} and \cite[Section 6]{Ro}. 
In order to illustrate historical developments we have added additional references following I. Efrat; namely to the crucial work of W. Magnus \cite{Mag1} and \cite{Mag}  related to some filtration of free groups. Magnus's paper \cite{Mag1}  initiated this research.
In our proof we use only \cite{Wi2}. 
The paper of Gr\"{u}n is not easy to read and it contains some gaps. Again following Efrat we refer the reader to a nice exposition of Gr\"{u}n's work (see \cite{Roh}).
  One should also mention that the results in \cite{Ko1} form extensions of the previous results of  \cite{Sko}, who proved these results for odd primes $p$, while Koch was able to extend them for all primes. Lazard's thesis \cite{Laz} contains a systematical approach to filtrations  on groups obtained from filtrations on rings building on the pioneering work of \cite{Mag}, \cite{Wi2} and a number of others. In  Chapter 5, in particular, Lazard deals with both the Zassenhaus filtration and the descending $p$-central series. Zassenhaus's original dimension groups were considered in \cite{Zas}. 
  
  The mathematical content of \cite{Ef2} and our results in Section 2 were obtained independently from each other and they seem to nicely complement to each other.
\begin{thm}
\label{thm:intersection}
Let $S$ be a free pro-$p$-group and $n\geq1$. Then
\begin{enumerate} 
\item[(a)] $S_{(n)}$ is the intersection of all kernels of linear representations $\rho\colon S\to {\U}_n(\F_p)$. 
\item[(b)] $S_{n}$ is the intersection of all kernels of linear representations $\rho\colon S\to {\U}_n(\Z_p)$.
\item[(c)] $S^{(n)}$ is the intersection of all kernels of linear representations $\rho\colon S\to {\U}_{k+1}(\Z/p^{n-k}\Z)$, where $k=1,\ldots,n-1$.
\end{enumerate}
\end{thm}
\begin{proof}
It suffices to consider that the case $S$ is finitely generated by a limit argument (\cite[Lemma 3.3.11]{RZ}). 

a) By Lemma~\ref{lem:3a}, one has \[ S_{(n)}\subseteq \bigcap \ker(\rho\colon S\to \U_n(\F_p)).\]

Now let $\sigma$ be any element in $S\setminus S_{(n)}$. We shall show that there exists a group representation $\rho:S\to \U_n(\F_p)$ such that $\sigma\not\in \ker\rho$ and then we are done. 

In fact, by Lemma~\ref{lem:1a} c), there exists a multi-index $I=(i_1,\ldots,i_{k})$ of length $k$ with $1\leq k\leq n-1$ such that $\epsilon_{I,\F_p}(\sigma)\not=0$.
By Lemma~\ref{lem:2a}, the map $\rho\colon S\to \U_{k+1}(\F_p)$ defined by 
 \[\rho(\tau)_{\mu\nu}=\epsilon_{(i_\mu,\ldots,i_{\nu-1}),\F_p}(\tau) \quad (\mu<\nu),\]
   is a continuous group homomorphism. Since $k+1\leq n$, we can embed $\U_{k+1}(\F_p)\hookrightarrow \U_n(\F_p)$ and obtain a homomorphism still denoted by $\rho$, 
\[\rho\colon S\to \U_{k+1}(\F_p)\hookrightarrow \U_n(\F_p).\]
Then $\rho(\sigma)_{1,k+1}=\epsilon_{I,\F_p}(\sigma)\not=0$. Therefore $\sigma\not\in\ker\rho$, as desired.

b) We proceed in the same way as in part a).

c) We proceed in the same way as in part a). However, for the convenience of the reader, we include a full proof here.
 By Lemma~\ref{lem:3a} c), one has \[ S^{(n)}\subseteq \bigcap_{k=1}^{n-1} \bigcap_{\rho} \ker(\rho\colon S\to \U_{k+1}(\Z/p^{n-k}\Z)).\]

Now let $\sigma$ be any element in $S\setminus S^{(n)}$. 
By Lemma~\ref{lem:1a} d), there exists a multi-index $I=(i_1,\ldots,i_{k})$ of length $k$ with $1\leq k\leq n-1$ such that $v(\epsilon_{I,\Z_p}(\sigma))< n-k$. This implies that $\epsilon_{I,\Z/p^{n-k}\Z}(\sigma)\not=0$.  
By Lemma~\ref{lem:2a}, the map $\rho\colon S\to \U_{k+1}(\Z/p^{n-k}\Z)$ defined by 
 \[\rho(\tau)_{\mu\nu}=\epsilon_{(i_\mu,\ldots,i_{\nu-1}),\Z/p^{n-k}\Z}(\tau) \quad (\mu<\nu),\]
   is a continuous group homomorphism.  
We have $\rho(\sigma)_{1,k+1}=\epsilon_{I,\Z/p^{n-k}}(\sigma)\not=0$. Therefore $\sigma\not\in\ker\rho$, as desired.
\end{proof}


Now let $S$ be an abstract (discrete) free group on  $d$ generators $x_1,\ldots,x_d$. Let $\Lambda$ be $\Z$ or $\Z/p^r\Z$ with $r\in \N$. We also have the Magnus homomorphism
\[
\psi\colon \Lambda[S] \to \Lambda \langle\langle X_1,\ldots,X_d\rangle\rangle, x_i\mapsto 1+X_i.
\]
Lemmas~\ref{lem:0a}-\ref{lem:3a} have their obvious counterparts in this case. These lemmas still hold true in the discrete setting by some obvious changes, for example, by replacing $\Z_p$ by $\Z$, completed group algebras by group algebras. Note also that the Magnus homomorphism $\psi$ is injective (see \cite[Lemma 8.1]{Mo} and \cite[Lemma 4.4]{Ko}).   The counterpart of Lemma~\ref{lem:1a} b) is ensured by a result of Witt \cite[Satz 11]{Wi2}, see also \cite[Proposition 8.5]{Mo}. By proceeding in the same way as in the proof of Theorem~\ref{thm:intersection}, we immediately obtain the following discrete version of Theorem~\ref{thm:intersection}.
\begin{thm} 
\label{thm:discrete free}
 Let $S$ be a finitely generated free group and $n\geq 1$. Then the following statements are true.
\begin{enumerate} 
\item[(a)] $S_{(n)}$ is the intersection of all kernels of linear  representations $\rho\colon S\to {\U}_n(\F_p)$.
\item[(b)] $S_{n}$ is the intersection of all kernels of linear representations $\rho\colon S\to {\U}_n(\Z)$.
\item[(c)] $S^{(n)}$ is the intersection of all kernels of linear representations $\rho\colon S\to {\U}_{k+1}(\Z/p^{n-k}\Z)$, where $k=1,\ldots,n-1$.
\end{enumerate}
\end{thm}
\section{Some results in unipotent representations of finite groups}

Let $K$ be a field of characteristic $p>0$ and let $s\geq 1$ be an integer. Let $X$ be the square matrix of size $n:=1+p^s$ having zero everywhere except for 1's in the positions $(i,i+1)$, $1\leq i \leq n-1$.  Note that $1+X$ is in $\U_n(K)$. We denote by $K[X]$ the $K$-algebra generated by $X$ in the full matrix algebra ${\rm Mat}_n(K)$.
\begin{lem}
Let the notation be as above. Then the following are true.
\begin{enumerate}
\item $X^n=0$ but $X^{n-1}\not=0$.
\item For $f(X)=a_l X^l + a_{l+1}X^{l+1}+\cdots\;\in K[X]$ with $a_l\not=0$, $f(X)$ is a unit in $K[X]$ if and only if $l=0$.
\item Every $K$-algebra automorphism $\varphi: K[X]\to K[X]$ is determined by the value of $\varphi(X)$ and has the form $\varphi(X)=Xf(X)$ with $f(X)$ a unit in $K[X]$. Conversely, if $f(X)$ is a unit in $K[X]$ then $\varphi(X)=Xf(X)$ defines a $K$-algebra automorphism $\varphi$ of $K[X]$.
\end{enumerate}
\end{lem}
\begin{proof} (1) and (2) are  straightforward. 

(3) Assume that $\varphi \colon K[X]\to K[X]$ is a $K$-algebra automorphism. Then we write 
\[
\varphi(X)=a_l X^l + a_{l+1}X^{l+1}+\cdots=: X^l f(X) \in K[X],
\]
where $a_l\not=0$. Then $f(X)$ is a unit in $K[X]$. It remains to show that $l=1$.
Since $X$ is not a unit, $\varphi(X)$  is not a unit either. Hence $l\geq 1$. But $l$ cannot be  $\geq 2$. Otherwise we would have
\[ 
\varphi(X^{n-1})= X^{l(n-1)}f(X)^{n-1}=0,
\]
because $l(n-1)\geq 2(n-1)>n$ as $n=p^s+1\geq 3$. This is a contradiction since $\varphi$ is injective and $X^{n-1}\not=0$. Therefore $l=1$, as desired.

Conversely, assume that $f(X)= a_0+a_1X+\cdots$, $a_0\not=0$, is a unit in $K[X]$. Let $\varphi\colon K[X]\to K[X]$ be a $K$-algebra endomorphism defined by $\varphi(X)=Xf(X)$. Then $\varphi(X^i)=X^if(X)^i$ and the matrix $M$ of $\varphi$ with respect to the base $\{1,X,\ldots,X^{n-1}\}$ of the $K$-vector space $K[X]$ is a lower triangular matrix with $1, a_0,a_0^2,\ldots, a_0^{n-1}$ on the diagonal. Since $\det (M)=\prod_{i=1}^{n-1}a_0^i\not=0$, $\varphi$ is an isomorphism as a $K$-linear map. We denote by $\psi$ its inverse as a $K$-linear map. Then we can check that $\psi$ is in fact  a $K$-algebra homomorphism. Therefore, $\varphi$ is a $K$-algebra automorphism. 
\end{proof}
The following Lemma~\ref{lem:centralizer} admits a simple direct proof which we shall omit.
\begin{lem}
\label{lem:centralizer}
The centralizer of $X$ in ${\rm Mat}_n(K)$ is $K[X]$.
\end{lem}

The following result is a generalization  of \cite[Lemma 4.2]{Ja}. In this lemma only the case of ${\rm char}(K) = 2$ was considered.
\begin{lem}
\label{lem:Janusz}
 Any $K$-algebra automorphism $\varphi$ of $K[X]$ is induced by conjugation with some upper triangular matrix $A\in {\rm GL}_n(K)$. For a given $\varphi$ there is  a unique such $A$ which has only zeros in the last column except for a $1$ in the $(n,n)$ position. Any other matrix inducing $\varphi$ has the form $AB$ for some unit $B$ in $K[X]$.
\end{lem}
\begin{proof}
 Let $\varphi(X)=Xf(X)$, where $f(X)$ is a unit in $K[X]$. Let $v_1,v_2,\ldots, v_n$ be the basis of $K^n$ such that
\[
 Xv_1=0; Xv_i=v_{i-1}, \;\forall 2\leq i\leq n.
\]
Define $A$ to be the matrix such that 
\[
 Av_i=f(X)^{n-i}v_i.
\]
Since $f(X)$ is a unit, $A$ is invertible.
And for $i\geq 2$, we have
\[
 AXv_i=Av_{i-1}=f(X)^{n-i+1}v_{i-1}=f(X)^{n-i+1}Xv_i= Xf(X)f(X)^{n-i}v_i=Xf(X)Av_i.
\]
The first and the last term are equal also for $i=1$. In fact, 
\[Xf(X)Av_1=Xf(X)f(X)^{p^s}v_1=f(X)f(X)^{p^s}Xv_1=0. \]
 Therefore $AX=Xf(X)A$ and thus
\[
 AXA^{-1}=\varphi(X).
\]

When $f(X)^{n-i}v_i$ is expressed in terms of  $v_j$, only $v_j$ with $j\leq i$ can have nonzero coefficients. (The shape of the polynomial $f(X)^{n-i}$ is not important for this observation.) Thus $A$ is upper triangular. Moreover, since $Av_n=v_n$, the last column of $A$ has only zero entries except for a 1 at the $(n,n)$ position. Also note that if $f(X)=1+aX+\cdots$ then $A$ is a unipotent upper triangular matrix.

Now assume that $A_0$ also induces $\varphi$. Then $A_0XA_0^{-1}=AXA^{-1}$. Hence $A^{-1}A_0X=XA^{-1}A_0$. Thus $A^{-1}A_0=B$, i.e., $A_0=AB$ with $B$ a unit in $K[X]$. Assume further that the last column of $A_0$ contains only zero except for a 1 at the $(n,n)$ position. Then $A_0v_n=v_n$. Hence
\[ Bv_n= A^{-1}A_0v_n=A^{-1}v_n=v_n.\]
Writing $B=a_0+a_1X+\cdots$  implies that $a_0=1$ and $a_i=0$ for all $1\leq i\leq n-1$. Hence $B=1$ and $A_0=AB=A$.
\end{proof}
\begin{lem}
\label{lem:2}
 Let the notation be as above. Let $B=1+X$. Let $k$ be a positive integer. Then
\begin{enumerate}
 
\item There exists a unique matrix $A$ in $\U_n(K)$ with only zero entries in the last column except at the $(n,n)$ position  such that 
\[
 ABA^{-1}=B^{1+p^k}.
\]
\item If $p=2$ then there exists a unique matrix $A$ in $\U_n(K)$ with only zero entries in the last column except at the $(n,n)$ position  such that 
\[
 ABA^{-1}=B^{-(1+2^k)}.
\]
\end{enumerate}
\end{lem}
\begin{proof}
We define a $K$-algebra automorphism $\varphi(X)$ of $K[X]$ by
\[
 \varphi(X):=
\begin{cases}
 X(1+X^{p^k-1}+X^{p^k}) &\text{ if we are in Part (1),}\\
 X(1+X^{2^k-1}+X^{2^k})(1+X)^{-(1+2^k)} &\text{ if we are in Part (2).}
\end{cases}
\]
Then  
\[
\varphi(1+X)= 
\begin{cases}
1+ X(1+X^{p^k-1}+X^{p^k})=(1+X)^{1+p^k} &\text{ if we are in Part (1),}\\
1+X(1+X^{2^k-1}+X^{2^k})(1+X)^{-(1+2^k)}=(1+X)^{-(1+2^k)} &\text{ if we are in Part (2).}
\end{cases} 
\]
By Lemma~\ref{lem:Janusz}, there exists an upper triangular matrix $A$ such that 
\[
 AXA^{-1}=\varphi(X),
\]
or equivalently, 
\[
 A(1+X)A^{-1}=\varphi(1+X).
\]
Also from the construction of $A$ as in the proof of Lemma~\ref{lem:Janusz}, we see that $A$ is a unipotent matrix. 
\end{proof}
When $p=2$, we also have the following result.  
\begin{lem}
\label{lem:3}
 Let the notation be as above. Let $B=1+X$.
 Assume that $p=2$.  Then there exists a unique matrix $A$ in $\U_n(K)$ with only zero entries in the last column except at the $(n,n)$ position  such that 
\[
 ABA^{-1}=B^{-1}.
\]
\end{lem}
\begin{proof}
 This is proved in \cite[page 154]{Ja}. 
\end{proof}

We will compute the order of matrix $A$ found in Lemma~\ref{lem:2}, Part (1). If $k>s$ then $(1+X)^{1+p^k}=(1+X)(1+X^{p^k})= 1+X$ and $A=I$ by the uniqueness of $A$. We consider the case $k\leq s$ in the following proposition.
\begin{prop}
\label{prop:order}
  Let the notation be as in Lemma~\ref{lem:2}, Part (1). Assume that $k\leq s$. Then the  order of $A$ is $p^{s+1-k}$.
\end{prop}
 First we need the following elementary lemma.
\begin{lem} Let $a,b,l$ be three integers with $l\geq 1$.  Assume that 
\[a\equiv b \bmod p^l \text{ and } a\not\equiv b\bmod p^{l+1}.\]
 Then $a^p\equiv b^p\bmod p^{l+1}$ and $a^p\not\equiv b^p\bmod p^{l+2}$.
\qed
\end{lem}
\begin{proof}[Proof of Proposition~\ref{prop:order}]
Since $\varphi(1+X)= A(1+X)A^{-1}=(1+X)^{1+p^k}$, we obtain $\varphi^{p^l}(1+X)= A^{p^l}(1+X)A^{-p^l}=(1+X)^{(1+p^k)^{p^l}}$, for all $l\geq 0$. 

On the other hand, since $(1+p^k)\equiv 1\bmod p^k$ and $(1+p^k)\not \equiv 1\bmod p^{k+1}$, we obtain
 \[
 (1+p^k)^{p^{s+1-k}}\equiv 1 \bmod p^{s} \text{ and } (1+p^k)^{p^{s-k}}\not \equiv 1\bmod p^s. 
 \]
Thus $ (1+X)^{(1+p^k)^{p^{s+1-k}}}=1+X$ and  $ (1+X)^{(1+p^k)^{p^{s-k}}}\not=1+X$ since $1+X$ is  of order $p^{s+1}$. Hence
\[ A^{p^{s+1-k}}(1+X)A^{-p^{s+1-k}}=(1+X)^{(1+p^k)^{p^{s+1-k}}}=1+X, 
\]
and
\[ A^{p^{s-k}}(1+X)A^{-p^{s-k}}=(1+X)^{(1+p^k)^{p^{s-k}}}\not=1+X.
\]
Therefore $A^{p^{s+1-k}}=1$ (by the uniqueness statement applied to $\varphi^{p^{s+1-k}}$ and $A^{p^{s+1-k}}$), and  $A^{p^{s-k}}\not=1$. Hence the order of $A$ is $p^{s+1-k}$.
\end{proof}
Let $p$ be a prime. The extra-special group of order $p^3$ and exponent $p^2$ is the group 
\[
M_{p^3} = \langle x, y \mid  x^{p^2} = y^p = 1, yxy^{-1} = x^{1+p}\rangle.
\]
\begin{cor}
\label{cor:embedding} Let $p$ be an odd  prime number and $K$ a field of characteristic $p$. 
 Let $M$ be the cyclic group $\Z/p^2\Z$ or $M_{p^3}$. The smallest number $n$ such that $M$ can be embedded in $\U_n(K)$ is $n=p+1$.
\end{cor}
\begin{proof}
 Note that every element in $\U_p(K)$ is of exponent $p$. Hence $M$ cannot be embedded in $\U_r(K)$ with $r\leq p$. 

If $M=\Z/p^2\Z$, we define an embedding of $M$ into  $\U_{p+1}(K)$ by sending $1$ to the matrix $B=1+X$.

If $M=M_{p^3}$ then we define a group homomorphism $\varphi\colon M \to \U_{p+1}(K)$ by sending $x$ to  matrix $B$ and $y$ to matrix $A$, where $A$ is a matrix  in $\U_{p+1}(K)$ such  that $ABA^{-1}=B^{1+p}$. The existence and uniqueness of $A$ is assured by Lemma~\ref{lem:2}. Furthermore the (group) order of $A$ is $p$. Then $\varphi$ is well-defined and $\varphi$ is an injection. 
\end{proof}
\section{Galois groups of maximal $p$-extensions of rigid fields}
Let $p$ be an odd prime. Let $F$ be a field containing a (fixed) primitive $p$-th root of unity. For each $a\in F^\times$,  let $\chi_a\in H^1(G_F,\F_p)= H^1(G_F(p),\F_p)$ be the  character associating to $a$ via the Kummer map $F^\times\to H^1(G_F,\F_p)=H^1(G_F(p),\F_p)$. 
\begin{defn}
\label{defn:p-rigid}
 An element $a\in F\setminus F^p$ is called {\it $p$-rigid} if $\chi_a\cup \chi_b=0$ implies $b\in a^i F^p$ for some $i\geq 0$. The field $F$ is called {\it $p$-rigid} if every element in $F\setminus F^p$ is $p$-rigid.
\end{defn}

For $h > 0$, let $\mu_{p^h}$  be the group of $p^h$ roots of unity. We also set $\mu_{p^\infty}$ to be the group of  all roots of unity of order $p^m$ for some $m \ge 0$. Finally we set $k \in \mathbb{N} \cup  \{ \infty \}$ to be the maximum of all $h \in \mathbb{N} \cup \{ \infty \}$ such that
$\mu_{p^h} \subseteq F$. Then we have the following theorem, see \cite[Theorem 4.10]{CMQ} and also \cite[Theorem 2]{Wa3} for part (2). 

\begin{thm}
\label{thm:CMQ}
Suppose $F$ is a $p$-rigid field and let $G=G_F(p)$. Then we have the following.
\begin{enumerate}
\item \[ G/G^{(n)}  =\begin{cases}
\left( \prod_{\mathcal{I}}  \mathbb{Z}/p^{n-1}\Z \right) \rtimes \mathbb{Z}/p^{n-1}\Z  \ \ &\text{if } k < \infty,\\
\prod_{\mathcal{I}}  \mathbb{Z}/p^{n-1}\Z \ \ &\text{if } k = \infty.
\end{cases}
\]
\item 
\[ G =\begin{cases}
\left( \prod_{\mathcal{I}}  \mathbb{Z}_p \right) \rtimes \mathbb{Z}_p  \ \ &\text{if } k < \infty,\\
\prod_{\mathcal{I}}  \mathbb{Z}_p \ \ &\text{if } k = \infty.
\end{cases}
\]
\end{enumerate}
Moreover when $k < \infty$ there exists a generator $\sigma$ of the outer factor $\mathbb{Z}/p^{n-1}\Z$ in $(1)$ and of the outer factor $\mathbb{Z}_p$ in $(2)$ such that for each  $\tau$ from the inner factor $\prod_{\mathcal{I}}  \mathbb{Z}/p^{n-1}\Z$ in $(1)$ and  each $\tau$ from the inner factor $\prod_{\mathcal{I}} \mathbb{Z}_p$ in $(2)$ we have
\[ \sigma \tau \sigma^{-1} = \tau^{p^k+1}.\]
\end{thm}

\begin{cor}
\label{cor:Zassenhaus quotient}
 Let $F$ be a $p$-rigid field and $G=G_F(p)$. Let $n\geq 2$ and let $s$ be the integer such that $p^{s-1}<n\leq p^{s}$. Then $G_{(n)}=G^{p^{s}}$ and 
 \[ G/G_{(n)} =\begin{cases}
\left( \prod_{\mathcal{I}}  \mathbb{Z}/p^{s}\mathbb{Z}\right) \rtimes \mathbb{Z}/p^{s}\mathbb{Z}  \ \ &\text{if } k < \infty,\\
\prod_{\mathcal{I}}  \mathbb{Z}/p^{s}\mathbb{Z} \ \ &\text{if } k = \infty.
\end{cases}
\]
Moreover when $k < \infty$ there exists a generator $\sigma$ of the outer factor $\mathbb{Z}/p^{s}\mathbb{Z}$  such that for each $\tau$ from the inner factor $\prod_{\mathcal{I}}\mathbb{Z}/p^{s+1}\mathbb{Z}$ we have
\[ \sigma \tau \sigma^{-1} = \tau^{p^k+1}.\]
\end{cor}

\begin{proof}
 The first statement $G_{(n)}=G^{p^s}$ is proved at the end of the paper \cite{CMQ}. The second statement follows from Theorem~\ref{thm:CMQ} and $G^{(s+1)}=G^{p^s}$ 
\cite[Remark 4.2]{CMQ}.
\end{proof}
\section{Kernel Unipotent Conjecture over odd rigid fields}

\begin{prop}
\label{prop:main} Let $k\geq 1$ be an integer. 
 Let $H$ be a pro-$p$-group. Assume that $H$ satisfies one of the following conditions.
\begin{enumerate}
\item \[ H\simeq (\prod_I \Z/p^{s+1}\Z). \]
 \item \[ H\simeq (\prod_I \Z/p^{s+1}\Z) \rtimes(\Z/p^{s+1}\Z)=:U\rtimes V, \]
  and there exists a generator $\sigma\in V$ for $V$ such that  $\sigma\tau\sigma^{-1}=\tau^{p^k+1}$, for all $\tau\in U$.
\item If $p=2$ and 
\[ H\simeq (\prod_I \Z/p^{s+1}\Z) \rtimes(\Z/p^{s+1}\Z)=:U\rtimes V, \]
  and there exists a generator $\sigma\in V$ for $V$ such that  $\sigma\tau\sigma^{-1}=\tau^{-(2^k+1)}$, for all $\tau\in U$.
\item If $p=2$ and 
\[ H\simeq (\prod_I \Z/p^{s+1}\Z) \rtimes(\Z/p^{s+1}\Z)=:U\rtimes V, \]
  and there exists a generator $\sigma\in V$ for $V$ such that  $\sigma\tau\sigma^{-1}=\tau^{-1}$, for all $\tau\in U$.
\end{enumerate}
Let $n=p^s+1$. Then for every $u\in H$, $u\not=1$, there exists a representation $H\to \U_n(\F_p)$ such that $\rho(u)\not=1$. 
\end{prop}

\begin{proof} Let $B:=1+X$ be as in Lemmas~\ref{lem:2} and \ref{lem:3}.
 
(1) We write $u=(u_i)_I\in U$ and let $C_i$ be a copy of $\Z/p^{s+1}\Z$ at the $i$-th coordinate in $U$. Then there exists $i_0$ such that $u_{i_0}$ is not the indentity element in $C_{i_0}$. Let $\tau$ be a generator of $C_{i_0}$ and let us write $u_{i_0}=\tau^a$ with $p^{s+1}\nmid a$. We define a representation $\rho\colon H\to \U_n(\F_p)$ by: $\rho(\tau)\mapsto B$, $\rho(C_i)=1$ for all $i\not=i_0$. Since $B^{p^{s+1}}=1$, $\rho$ is a well-defined homomorphism. Moreover $\rho(u)=\rho(u_{i_0})=B^a\not=1$.  
\\
\\
(2)-(4):    For convenience, for any element $g$ in a group $G$, we write $\varphi(g)=g^{1+p^k}$,  $g^{-(1+2^k)}$, or $g^{-1}$ if we are in Part (2), Part (3), or Part (4) respectively. 

In our construction of suitable homomorphisms $\rho \colon U\rtimes V\to \U_4(\F_p)$ below, we use the fact that the semidirect presentations (2), (3) and (4) allow us to pick up generators $\tau_i$, $i\in I,$ of $U$ and a generator $\sigma$ of $V$ such that  defining relations for $U\rtimes V$, besides orders of generators and trivial commutators $[\tau_i,\tau_j]$, are just relations $\sigma\tau_i\sigma^{-1}=\varphi(\tau_i)$, $i\in I$.
\\
\\
We write $x=uv$, with $u\in U$, $v\in V$.
\\  
\\
 {\bf Case 1:} $v\not =1$.  We define a group homomorphism $\rho: H\to \U_n(\F_p)$ as follow: $\rho(\tau)=1$ for all $\tau\in U$ and $\rho(\sigma)=B$. This is a well-defined homomorphism since it is clear that $B1B^{-1}=1=\varphi(1)$ and that $B^{p^{s+1}}=1$. Writing $v=\sigma^a$ for some $a\in\Z$ such that $p^{s+1}\nmid a$ ($v\not=1$ in $V$). Since $B$ is of order $p^{s+1}$, $\rho(x)=\rho(v)=B^a\not=1$, as desired.
\\
\\
{\bf Case 2:} $v=1$. Then $u\not =1$. 

By Lemmas~\ref{lem:2}  and~\ref{lem:3}, there exists $A\in \U_n(\F_p)$ such that $ABA^{-1}=\varphi(B)$. We can define a representation $\rho:U\rtimes V\to \U_n(\F_p)$ such that $\rho(x)=\rho(u)\not=1$ as follows.
We write $u=(u_i)_I\in U$ and let $C_i$ be a copy of $\Z/p^{n+1}\Z$ at the $i$-th coordinate in $U$. Then there exists $i_0$ such that $u_{i_0}$ is not the indentity element in $C_{i_0}$. Let $\tau_0$ be a generator of $C_{i_0}$, and let us write $u_{i_0}=\tau_0^a$ with $p^{s+1}\nmid a$. We define $\rho$ by: $\rho(\tau_0)\mapsto B$, $\rho(C_i)=1$ for all $i\not=i_0$ and $\rho(\sigma)=A$. Since $B^{p^{s+1}}=1$, we see that $\rho(\tau^{p^{s+1}})=1$ for all $\tau\in U$. We also have $\rho(\sigma^{p^{s+1}})=A^{p^{s+1}}=1$ by Proposition~\ref{prop:order}. Since $ABA^{-1}=\varphi(B)$ and hence $A B^rA^{-1}=\varphi(B)^r=\varphi(B^r)$ for all $r$, we see that $\rho(\sigma \tau\sigma^{-1})=\varphi(\rho(\tau))=\rho(\varphi(\tau))$ for all $\tau\in U$. Therefore,  we conclude that $\rho$ is a well-defined homomorphism. Moreover $\rho(u)=\rho(u_{i_0})=B^a\not=1$. 
\end{proof}
\begin{thm}
\label{thm:main}
 Let $p$ be a prime number and let $F$ be a rigid field. Let $G=G_F(p)$ be the Galois group of the maximal $p$-extension of $F$. Then for any natural number $n$,
\[
G_{(n)}=\bigcap_{\rho} \ker(\rho:G\to \U_n(\F_p)),
\]
where $\rho$ runs over the set of all continuous homomorphisms $G\to \U_n(\F_p)$.
\end{thm}
\begin{proof}
By Corollary~\ref{cor:Zassenhaus quotient}, for any integer $s\geq 0$ we have $G_{(p^s+1)}=G_{(p^s+2)}=\cdots=G_{(p^{s+1})}$. Thus it is enough to consider the case $n=p^s+1$. Then the statement follows from Corollary~\ref{cor:Zassenhaus quotient} and  Proposition~\ref{prop:main}.
\end{proof}

\subsection*{Demushkin groups of rank 2} (In this subsection we also consider $p=2$.)
Recall (see e.g., \cite{La,Se1}) that a pro-$p$-group $G$ is said to be a Demushkin group if
\begin{enumerate}
\item $\dim_{\F_p} H^1(G,\F_p)<\infty,$ 
\item $\dim_{\F_p} H^2(G,\F_p)=1,$
\item  the cup product $H^1(G,\F_p)\times H^1(G,\F_p)\to H^2(G,\F_p)$ is a non-degenerate bilinear form.
\end{enumerate}
We call $d:=\dim_{\F_p} H^1(G,\F_p)$ the {\it rank} of $G$. When the rank $d=2$ then $G$ has  the following presentation (see \cite[page 147]{Se1}): $G$ is isomorphic to a pro-$p$-group on two generators $x,y$ subject to one relation in one of the following types
\[
\begin{aligned}
\text{(Type 1)}\quad  yxy^{-1}&=x,  \\ 
 \text{(Type 2)}\quad  yxy^{-1}&=x^{1+q} \text{ for some $q$ with ($q=p^k\geq p$ if  $p>2$) and $(q=2^k\geq 4$ if $p=2$)},\\
 \text{(Type 3)}\quad yxy^{-1}&=x^{-1} \text{ and } p=2,\\
 \text{(Type 4)}\quad yxy^{-1}&=x^{-(1+m)} \text{ for some  $m=2^k\geq 4$ and } p=2. 
\end{aligned} 
\]
Let $F$ be a local field of residue characteristic different from $p$. Assume that $F$ contains a primitive $p$-th root of unity. Then the Galois group $G_F(p)$ of the maximal $p$-extension of $F$ is a Demushkin group of rank $2$.   
\begin{lem}
\label{lem:Demushkin rank 2}
Let $G$ be a Demushkin pro-$p$-group generated by $x,y$ and subject to one relation as above. Let $n\geq 2$ and let $s$ be a unique integer such that $p^{s-1}<n\leq p^s$. Then $G_{(n)}=G^{p^s}$ and 
\[
G/G_{(n)}= 
\begin{cases}
 \Z/p^s\Z \times \Z/p^s\Z \text{ if we are in Type 1}, \\ 
\Z/p^s\Z \rtimes \Z/p^s\Z \text { if we are in Type 2},\\ 
 \Z/2^s\Z \rtimes \Z/2^s\Z \text { if we are in Type 3},\\
  \Z/2^s\Z \rtimes \Z/2^s\Z. \text{ if we are in Type 4}.
  \end{cases}
\]
Moreover, let $\sigma$ be the image of $x$ in the outer factor $\Z/p^{s}$, then for each $\tau$ from the inner factor $\Z/p^s\Z$ we have
\[
\begin{aligned}
\sigma\tau\sigma^{-1}= \tau^{1+p^k} \text{ if we are in Type 2},\\
\sigma\tau\sigma^{-1}= \tau^{-1} \text{ if we are in Type 3},\\
\sigma\tau\sigma^{-1}= \tau^{-(1+2^k)} \text{ if we are in Type 4}.
\end{aligned}
\]
\end{lem}
\begin{proof} If we are in Types (1)-(2), one can see that $[G,G]\subseteq G^p$ and if we are in Type 4, one has $[G,G]\subseteq G^4$. Hence $G$ is  {\it powerful} (\cite[Chapter 3, Definition 3.1]{DDMS}). Hence in these cases, $G_{(n)}=G^{p^s}$ by \cite[Exercise 4, p. 289]{DDMS}. If we are in Type 3, one can show that $G_{(n)}=G^{2^s}$ by induction on $n$.

Then the remaining statements are  straightforward.
\end{proof}
\begin{prop}
\label{prop:Demushkin rank 2}
 Let $G$ be a pro-$p$-group. Assume that $G$ is a Demushkin group of rank $2$. 
Then for any natural number $n$,
\[
G_{(n)}=\bigcap_{\rho}\ker(\rho:G\to \U_n(\F_p)),
\]
where $\rho$ runs over the set of all continuous homomorphisms $G\to \U_n(\F_p)$. 
\end{prop}
\begin{proof}
This proof is similar to that of Theorem~\ref{thm:main}, using Lemma~\ref{lem:Demushkin rank 2} instead of Corollary~\ref{cor:Zassenhaus quotient}.
\end{proof}
\begin{rmk}
If the rank of a Demushkin group $G$ is 1 then $G$ is a cyclic group of order 2, and hence the Kernel $n$-Unipotent Conjecture is true for $G$.
\end{rmk}
\begin{rmk}
 Every Demushkin group $G$ of rank 2 is realizable as $G_F(p)$ for some field $F$. In fact, if its relation is of Type 1, then $G\simeq G_F(2)$ with $F=\C((X))((Y))$ (see \cite[Corollary 3.9, part (2)]{Wa1}). Now assume that the relation is of Type 2. By Dirichlet's theorem on primes in arithmetic progressions, there is a prime number $\ell$ such that $\ell= p^k+1 + p^{k+1}a$, for some $a\in \Z$. Then by \cite[Exercise 2, page 98]{Se}, the Galois group $G_{\Q_\ell}(p)$ is isomorphic to a Demushkin group on two generators $x,y$ with one relation \[yxy^{-1}= x^{1+p^k}.\]
Similarly, assume that we are given a pro-2 Demushkin group $G$ with generators $x,y$ such that  the relation is of Type 4: 
\[
yxy^{-1}= x^{-(1+2^k)}.
\]
Then by choosing a prime number $\ell$ such that $\ell =2^k-1+ 2^{k+1}a$, for some $a\in \Z$, we see that the Galois group $G_{\Q_\ell}(2)$ is isomorphic to $G$ (see \cite[Exercise 2, page 98]{Se}).  For the case that our given pro-$2$ Demushking group is of Type 3,  we refer the reader to \cite[Table 5.2, Example 5]{JW}. 
Unfortunately, in this table there is a reference to Remark 2.6 which is missing in the paper in \cite{JW}. However, for our purposes, it suffices to observe that the required field is $M((T))$. Where $M$ is a maximal algebraic extension of $\R(X)(\sqrt{-(1+X^2)})$, subject to the condition that $\sqrt{-1}\not\in M$.
\end{rmk}

\section{Comparison between filtrations}
\label{sec:6}
 Let $G$ be a pro-$p$-group. Let us consider two filtrations on $G$: the descending $p$-central series $(G^{(i)})$ and the $p$-Zassenhaus filtration $(G_{(i)})$ of $G$. 
In general we always have $G^{(i)}\subseteq G_{(i)}$, for all $i$ and all $p$ and $G^{(3)}=G_{(3)}$ if $p=2$. 
From the very definition, we get $G_{(p+1)}\subseteq G^{(3)}$. Indeed, we have
\[
 G^{(3)}=(G^{(2)})^p[G^{(2)},G] \text{ and } G^{(p+1)}=(G_{(2)})^p\prod_{i+j=p+1}[G_{(i)},G_{(j)}].
\]
Since $G_{(2)}=G^{(2)}$ and $[G_{(i)},G_{(p+1-i)}]\subseteq [G_{(2)},G]=[G^{(2)},G]$ for all $i=1,\ldots,p$, we get $G_{(p+1)}\subseteq G^{(3)}$. For convenience, we introduce another filtration $G_{<n>}$, called the {\it kernel filtration},  on  group $G$ which is defined by
\[
 G_{<n>}:= \bigcap_{\rho} \ker(\rho\colon G\to\U_{p+1}(\F_p)),
\]
where the intersection is taken over the collection of all (continuous) group homomorphisms $\rho\colon G\to \U_{p+1}(\F_p)$. Note  that by  \cite[Lemma 3.6]{MT1}, one always has 
\[
 G_{(n)}\subseteq G_{<n>}.
\]

The kernel conjecture can be restated as the following:
 \begin{conj}
  Let $F$ be a field of characteristic $\not=p$, containing a primitive $p$th root of unity. Let $G=G_F(p)$ be the Galois group of the maximal pro-$p$-extension of $F$. Then for any $n$, one has
\[
 G_{(n)}=G_{<n>}.
\]
 \end{conj}

Inspired by Corollary~\ref{cor:embedding} and the kernel conjecture, we have the following result which is interesting in its own right.
\begin{prop}
\label{prop:smallest}
 Let $p$ be  prime and let $F$ be a field of characteristic $\not=p$, which contains a primitive $p$-th root of 1. Let $G=G_F(p)$ be the Galois group of the maximal $p$-extension of $F$.
\begin{enumerate}
\item We have \[
G_{(p+1)}\subseteq G_{<p+1>} \subseteq G^{(3)}.
\]
\item  Assume further that $G$ is not isomorphic to the trivial group $1$ or the cyclic group $\Z/2\Z$ then  $p+1$ is the smallest integer $n$ with the property that
 $G_{(n)}\subseteq G^{(3)}$. 
 \end{enumerate}
\end{prop}
\begin{proof} a) 
  By \cite[Main Theorem]{EM1}, $G^{(3)}$ is the intersection of the kernels of all homomorphisms from $G$ to $\Z/p^2\Z$ or to $M_{p^3}$. Then  Corollary~\ref{cor:embedding} implies that 
\[
 G^{(3)}\supseteq \bigcap_{\rho} \ker(\rho\colon G\to\U_{p+1}(\F_p))=G_{<p+1>}.
\]
 \\
b) We first show that if $G$ is not isomorphic to $1$ or $\Z/2\Z$ then  $G^{[3]}:=G/G^{(3)}$ is of exponent exactly $p^2$. It is clear that every element of $G^{[3]}$ is of order at most $p^2$, so it remains to show that $G^{[3]}$ contains an element of order $p^2$.

First we treat the case  $p>2$. Since $G$ is a non-trivial pro-$p$-group, one has ${\rm Hom}(G,\F_p)\simeq {\rm Hom}(G/G^{(2)},\F_p)\not=1$ (see for example \cite[Corollary 2, page 19]{Se}). In particular we can take an element $x\in G\setminus G^{(2)}$. Then $x$ has order exactly $p^2$ in $G^{[3]}$ because $x^{p^2}\in G^{(3)}$ and by \cite[Proposition 12.3]{EM1} (see also \cite[Theorem A.3]{BLMS}) we know that every element of $G/G^{(3)}$ of order $p$ is in fact in  $G^{(2)}/G^{(3)}$. 

Now we consider the case $p=2$ and assume that $G^{[3]}=G/G^{(3)}$ is of exponent (at most) 2. This implies in particular that $G^{[3]}$ is abelian and  is isomorphic to $(\Z/2\Z)^I$. 
By the classification of  abelian $W$-groups \cite[Theorems  3.12 and 3.13]{MS1} (note that the $W$-groups defined there coincide with groups of the form $G_F^{[3]}$), then $G^{[3]}$, which is also denoted by ${\mathcal G}_F$ in \cite{MS1},  is the trivial group or isomorphic to $\Z/2Z$. 
One can check directly that this means that $F$ is $F(2)$, or $F$ has a unique quadratic extension $L=L(2) =F(2)$. (Here for a field $F$, we denote $F(2)$ the maximal $2$-extension of $F$.) Hence $G$ itself is trivial or isomorphic to$\ Z/2\Z$.
 This contradicts  our assumption. Hence $G^{[3]}$ is of exponent $4$. 

On the other hand, from the definition of the $p$-Zassenhaus filtration, we see that $G/G_{(p)}$ is of exponent (at most) $p$. Therefore there cannot exist any surjections from 
\[
G/G_{(p)} \to G/G^{(3)}.
\]
In particular, $G^{(3)}$ does not contain $G_{(p)}$ and hence also does not contain $G_{(n)}$ for all $1\leq n\leq p$. 
 \end{proof}
\begin{rmk} a) The proof of Proposition~\ref{prop:smallest} part b) shows that for a non-trivial pro-$p$-group $G$, which is also not isomorphic to $\Z/2\Z$,  if $G/G^{(3)}$ is of exponent $p$ then $G\not\cong G_F(p)$ for any field $F$ containing a primitive $p$-th root of unity. Note also that if $G/G^{(3)}$ is of exponent $p$ then $G^{(3)}$ in fact contains $G_{(p)}$.

b) There are examples of pro-$p$-groups $G$ such that $G$ does not satisfy the conclusion in part a) of Proposition~\ref{prop:smallest} (see Appendix).  Then such pro-$p$-groups also cannot   be realized as $G_F(p)$  for any field $F$ containing a primitive $p$-th root of unity.
\end{rmk}

\section{Review of Massey products in the cohomology of profinite groups}
\label{sec:Review Massey}

Let $p$ be a prime number. Let  $G$ be a profinite group. We consider $\F_p$ as a trivial discrete $G$-module. Let $\sC^\bullet(G,\F_p,\delta)$ be the standard inhomogeneous continuous cochain complex of $G$
with coefficients in $\F_p$ \cite[Ch.\ I, \S2]{NSW}. We write $H^i(G,\F_p)$ for the corresponding cohomology groups.

We shall assume that $\alpha_1,\ldots,\alpha_n$ are elements in $H^1(G,\F_p)$.

\begin{defn}
 A collection $M=(a_{ij})$, $1\leq i<j\leq n+1$, $(i,j)\not=(1,n+1)$, of elements of $\sC^1(G,\F_p)$ is called a {\it defining system} for the {\it $n$-fold  Massey product} $\langle \alpha_1,\ldots,\alpha_n\rangle$ if the following conditions are fulfilled:
\begin{enumerate}
\item $a_{i,i+1}$ represents $\alpha_i$.
\item $\delta \alpha_{ij}= \sum_{l=i+1}^{j-1} a_{il}\cup a_{lj}$ for $i+1<j$.
\end{enumerate}
Then $\sum_{k=2}^{n} a_{1k}\cup a_{k,n+1}$ is a $2$-cocycle.
Its  cohomology class in $H^2(G,\F_p)$  is called the {\it value} of the product relative to the defining system $M$,
and is denoted by $\langle \alpha_1,\ldots,\alpha_n\rangle_M$.

The product $\langle \alpha_1,\ldots,\alpha_n\rangle$ itself is the subset of $H^2(G,\F_p)$ consisting of all  elements which can be written in the form $\langle \alpha_1,\ldots, \alpha_n\rangle_M$ for some defining system $M$. 
\end{defn}

From the condition (2) in the above definition, we see that if $\langle \alpha_1,\ldots,\alpha_n\rangle$ is defined, then  every $k$-fold Massey product $\langle \alpha_i,\ldots,\alpha_{i+k-1}\rangle$ with $2\leq k<n$ is defined and contains 0.


Recall that $\U_{n+1}(\F_p)$ is the group of all upper-triangular unipotent $(n+1)\times(n+1)$-matrices
with entries in  $\F_p$.
Let $Z_{n+1}(\F_p)$ be the subgroup of all such matrices with all off-diagonal entries
being $0$ except at position $(1,n+1)$.
We may identify $\U_{n+1}(\F_p)/Z_{n+1}(\F_p)$ with the group $\bar\U_{n+1}(\F_p)$
of all upper-triangular unipotent $(n+1)\times(n+1)$-matrices with entries over $\F_p$
with the $(1,n+1)$-entry omitted.

For a representation $\rho\colon G\to \U_{n+1}(\F_p)$ and $1\leq i< j\leq n+1$,
let $\rho_{ij}\colon G\to \F_p$
be the composition of $\rho$ with the projection from $\U_{n+1}(\F_p)$ to its $(i,j)$-coordinate.
We use  similar notation for representations $\bar\rho\colon G\to\bar\U_{n+1}(\F_p)$.
Note that $\rho_{i,i+1}$ (resp., $\bar\rho_{i,i+1}$) is a group homomorphism.

\begin{thm}[{\cite[Theorem 2.4]{Dwy75}}]
\label{thm:Dwyer}
Let $\alpha_1,\ldots,\alpha_n$ be elements of $H^1(G,\F_p)$. There is a one-one correspondence $M\leftrightarrow \bar\rho_M$ between defining systems $M$ for $\langle \alpha_1,\ldots,\alpha_n\rangle$ and group homomorphisms $\bar\rho_M:G\to \bar{\U}_{n+1}(\F_p)$ with $(\bar\rho_M)_{i,i+1}= -\alpha_i$, for $1\leq i\leq n$.

Moreover $\langle \alpha_1,\ldots,\alpha_n\rangle_M=0$ in $H^2(G,\F_p)$ if and only if the dotted arrow exists in the following  commutative diagram
\[
\xymatrix{
& & &G \ar@{->}[d]^{\bar\rho_M} \ar@{-->}[ld]\\
0\ar[r]& \F_p\ar[r] &\U_{n+1}(\F_p)\ar[r] &\bar{\U}_{n+1}(\F_p)\ar[r] &1.
}
\]
\end{thm}

Explicitly, the one-one correspondence in Theorem~\ref{thm:Dwyer} is given by: For a defining system $M=(a_{ij})$ for $\langle \alpha_1,\ldots,\alpha_n\rangle$, $\bar\rho_M\colon G\to \bar{\U}_{n+1}(\F_p)$ is given by letting $(\bar\rho_M)_{ij}=-a_{ij}$.

\begin{cor}
 Let $\rho\colon G\to \U_{n+1}(\F_p)$ be a group homomorphism, then the $n$-fold Massey product $\langle -\rho_{1,2},\ldots,-\rho_{n,n+1}\rangle$ is defined and contains 0.
\end{cor}

\section{Vanishing of Massey products over odd rigid fields}
\begin{thm}
\label{thm:p>2}

Let $p$ be a prime number and $n\geq 3$ an integer. Let $F$ be a field. Assume that $F$ contains a primitive $p$-th root of unity if ${\rm char}(F)\not=p$ and assume further that if $p=2$ then $-1$ is in $F^2$.  Let $G$ be the absolute Galois group $G_F$ of $F$ or its maximal pro-$p$ quotient $G_F(p)$. Then for any  $\chi \in H^1(G,\F_p)$, the $n$-fold Massey product $\langle \chi,\ldots,\chi\rangle$  is defined and contains 0.
\end{thm}
It is enough to consider the case $G=G_F(p)$. Also if ${\rm char}(F)=p$ then by a result of Witt (\cite{Wi}, \cite[Chapter II, \S2, Corollary 1]{Se1}), we know that $G_F(p)$ is a free pro-$p$-group. Hence, $\langle \chi,\ldots,\chi\rangle=0$. 

So we may assume that ${\rm char}(F)\not=p$, and let us fix a primitive $p$-th root of unity $\xi$. 
Then $\chi=\chi_a$ for some $a\in F^\times$, 
where $\chi_a\in H^1(G_F,\F_p)= H^1(G_F(p),\F_p)$ is  the  character associating to $a$ via the Kummer map $F^\times\to H^1(G_F,\F_p)=H^1(G_F(p),\F_p)$.

Also it suffices to consider the case $n=p^s$ for some integer $s$ since if the $m$-fold Massey product $\langle \chi_a,\ldots,\chi_a\rangle$ is defined, then   for all $2\leq n<m$, all $n$-fold Massey products $\langle\chi_a,\ldots,\chi_a\rangle$ are defined and contain 0. Also we can assume $a\not \in F^p$.

 For each integer $r\geq 1$, we choose a primitive $p^r$-th root of unity $\zeta_{p^s}$ in the way that  $\zeta_{p^{r+1}}^p=\zeta_{p^r}$.
 
\begin{proof} 
{\bf Case 0}: First we assume that $\zeta_{p^{s+1}}$ is in $F$. Then the polynomial  $X^{p^{s+1}}-a$ is irreducible over $F$ since $a \not\in (F^\times)^p$. Let $L=F(a^{1/p^{s+1}})$. Then $L/F$ is a cyclic extension of order $p^{s+1}$ whose Galois group generated by $\sigma$, where $\sigma$ defined by $\sigma(a^{1/p^{s+1}})=\zeta_{p^{s+1}}a^{1/p^{s+1}}$. We define a representation $\varphi$ from ${\rm Gal}(L/F)\to \U_{p^s+1}(\F_p)$ by $\varphi(\sigma)=B=1+X$, where $X$ is the matrix as in the beginning of Section 3. Then $\varphi$ is well-defined  since $B^{p^{s+1}}=1$. (In fact $\varphi$ is an isomorphism to its image.)  
Let $\rho:{\rm Gal}(F)\to \U_{p^s+1}(\F_p)$ be the composite ${\rm Gal}(F) \twoheadrightarrow {\rm Gal}(L/F)\stackrel{\varphi}{\to} \U_{p^s+1}(\F_p)$.
We claim that $\rho_{i,i+1}=\chi_a$, for $i=1,\ldots,p^s$. 
Indeed, since both maps $\rho$ and $\chi_a$ factor through the quotient ${\rm Gal}(L/F)$, it suffices to check on the generator $\sigma$ of ${\rm Gal}(L/F)$. 
Since $\sigma(a^{1/p})= \sigma(a^{1/p^{s+1}})^{p^s}=a^{1/p}\zeta_p$, one has
\[
 \varphi_{i,i+1}(\sigma)=1=\chi_a(\sigma).
\]
Therefore $\rho_{i,i+1}=\chi_a$. It implies that  $\langle -\chi_a,\ldots,-\chi_a\rangle$ contains 0 and hence $\langle \chi_a,\ldots,\chi_a\rangle$ also contains 0.
\\
\\
From now on, we assume that $\zeta_{p^{s+1}}$ is not in $F$. Let $p^k$ be the largest index such that $\zeta_{p^k}$ is in $F$, $1\leq k\leq s$. 
\\
\\ 
{\bf Case 1:} Assume that the polynomial $X^{p^{s+1}}-a$ is irreducible over $F(\zeta_{p^{s+1}})$. Then the extension $L:=F(\zeta_{p^{s+1}}, a^{1/p^{s+1}})$ is Galois over $F$. 
 Define two automorphisms $\sigma,\tau\in {\rm Gal}(L/F)$ as follows: 
\[ 
\begin{aligned}
\tau\colon a^{1/p^{s+1}}\mapsto \zeta_{p^{s+1}} a^{1/p^{s+1}} \text{ and }\tau\colon \zeta_{p^{s+1}} \mapsto \zeta_{p^{s+1}};\\
\sigma\colon \zeta_{p^{s+1}}\mapsto \zeta_{p^{s+1}}^{1+p^k} \text{ and } \sigma\colon a^{1/p^{s+1}} \mapsto a^{1/p^{s+1}}.
\end{aligned}
\]
Then we have 
\[
\begin{aligned}
&\sigma\tau\sigma^{-1}(\zeta_{p^{s+1}})=\zeta_{p^{s+1}};\\
&\sigma\tau\sigma^{-1}(a^{1/p^{s+1}})=\sigma\tau(a^{1/p^{s+1}})= \sigma (\zeta_{p^{s+1}}) a^{1/p^{s+1}}=\zeta_{p^{s+1}}^{1+p^k} a^{1/p^{s+1}}=\tau^{1+p^k}(a^{1/p^{s+1}}).
\end{aligned}
\]
Therefore $\sigma\tau\sigma^{-1}=\tau^{1+p^k}$ (*). 
In the group ${\rm Gal}(L/F)$, the order of $\tau$  is $p^{s+1}$ and that of $\sigma$ is $p^{s+1-k}$ and the group ${\rm Gal}(L/F)$ is presented as
\[
{\rm Gal}(L/F)=\big \langle \sigma, \tau\mid \tau^{p^{s+1}}=\sigma^{p^{s+1-k}}=1, \sigma\tau\sigma^{-1}=\tau^{1+p^k}\big \rangle. 
\]
Let $B=1+X$ and $A$ in $\U_{p^s+1}(\F_p)$ be the matrices as in Lemma~\ref{lem:2} so that $ABA^{-1}=B^{1+p^k}$ (*).
We define a homomorphism $\varphi$ from ${\rm Gal}(L/F)$ to $\U_{p^s+1}(\F_p)$ by letting $\varphi(\tau)=B$ and $\varphi(\sigma)=A$. Then $\varphi$ is indeed well-defined because $(*)$ holds, ${\rm ord}(B)=p^{s+1}$ and ${\rm ord}(A)=p^{s+1-k}$ by Proposition~\ref{prop:order}.
 (In fact, $\varphi$ is an isomorphism to its image.) Let $\rho:{\rm Gal}(F)\to \U_{p^s+1}(\F_p)$ be the composite ${\rm Gal}(F) \twoheadrightarrow {\rm Gal}(L/F)\stackrel{\varphi}{\to} \U_{p^s+1}(\F_p)$.
We claim that $\rho_{i,i+1}=\chi_a$, for each $i=1,\ldots, p^s$. 
Indeed, since both maps $\rho$ and $\chi_a$ factor through the quotient ${\rm Gal}(L/F)$, it suffices to check on the generators of ${\rm Gal}(L/F)$. 
Since $\tau(a^{1/p})=(\tau(a^{1/p^{s+1}}))^{p^s}=\zeta_{p} a^{1/p}$, we get
\[
 \varphi_{i,i+1}(\tau)=1=\chi_a(\tau).
\]
On the other hand, since $\sigma(a^{1/p})=a^{1/p}$ and the nearby diagonal entries of the matrix $A$ are 0, we get
\[
 \varphi_{i,i+1}(\sigma)=0=\chi_a(\sigma).
\]
Therefore $\rho_{i,i+1}=\chi_a$. This implies that  $\langle -\chi_a,\ldots,-\chi_a\rangle$ contains 0 and hence $\langle \chi_a,\ldots,\chi_a\rangle$ also contains 0.
\\
\\
{\bf Case 2:} Now assume that  $X^{p^{s+1}}-a$ is reducible over $F(\zeta_{p^{s+1}})$. Then this implies that $a$ is in $F(\zeta_{p^{s+1}})^p$ if $p>2$ and it is in $-4F(\zeta_{2^{s+1}})^4\subseteq F(\zeta_{2^{s+1}})^2$ if $p=2$ by \cite[Chapter VI, \S 9, Theorem 9.1]{Lan}. So in any case, we always have $a$  in $F(\zeta_{p^{s+1}})^p$. 
 Since $F(\zeta_{p^{s+1}})/F$ is a cyclic extension (of degree $p^{s+1-k}$),  a $p$-th root $a^{1/p}$ of $a$ has to be in $F(\zeta_{p^{k+1}})$. 
By Kummer theory, we get  $a\in \langle [\zeta_{p^{k}}]\rangle\subseteq F^\times/{F^\times}^p$.  
Thus $a=\zeta_{p^{k}}^m\bmod {F^\times}^p$ for some $0<m<p$. 
Hence $\chi_a=m \chi_{\zeta_{p^k}}$. By the linearity of Massey products (see e.g. \cite[Lemma 6.2.4]{Fe}), to show $\langle \chi_a,\ldots,\chi_a\rangle$ is defined and contains 0,  it suffices to show that $\langle \chi_{\zeta_{p^k}},\ldots, \chi_{\zeta_{p^k}}\rangle$ is defined and contains 0, i.e., we can assume from the beginning that $a=\zeta_{p^k}$. 

Now let $L=F(\zeta_{p^{s+1+k}})$. Then $L/F$ is a cyclic extension of order $p^{s+1}$ whose Galois group is generated by $\sigma$, where $\sigma(\zeta_{p^{s+1+k}})=\zeta_{p^{s+1}}.\zeta_{p^{s+1+k}}=\zeta_{p^{s+1+k}}^{1+p^k}$. We define a representation $\varphi$ from ${\rm Gal}(L/F)\to \U_{p^{s+1}}(\F_p)$ by $\varphi(\sigma)=B=1+X$. Then $\varphi$ is well-defined.  
Let $\rho:{\rm Gal}(F)\to \U_{p^s+1}(\F_p)$ be the composite ${\rm Gal}(F) \twoheadrightarrow {\rm Gal}(L/F)\stackrel{\varphi}{\to} \U_{p^s+1}(\F_p)$.
We claim that $\rho_{i,i+1}=\chi_a$, for each $i=1,\ldots,p^s$. 
Indeed, since both maps $\rho$ and $\chi_a$ factor through the quotient ${\rm Gal}(L/F)$, it suffices to check on the generator $\sigma$ of ${\rm Gal}(L/F)$. 
Since $\sigma(a^{1/p})= \sigma(\zeta_{p^{s+1+k}}^{p^s})=(\zeta_{p^{s+1}}.\zeta_{p^{s+1+k}})^{p^s}=a^{1/p}\zeta_p$, one has
\[
 \varphi_{i,i+1}(\sigma)=1=\chi_a(\sigma).
\]
Therefore $\rho_{i,i+1}=\chi_a$. This implies that  $\langle -\chi_a,\ldots,-\chi_a\rangle$ contains 0 and hence $\langle \chi_a,\ldots,\chi_a\rangle$ also contains 0.
\end{proof}
\begin{ex}
\label{ex:p-cyclic}
 Let $p$ be a prime number and $G=\Z/p\Z$. Let $\chi\in H^1(G,\F_p)$ be the identity map. In \cite[Example 4.6]{MT1}, it is shown that if $p>2$ then the $p$-fold Massey product $\langle \chi,\ldots,\chi\rangle$ is defined but does not contain 0. 

Now assume that $p=2$, we claim that the $4$-fold Massey product $\langle\chi,\chi,\chi,\chi\rangle$ is not defined. 
For a contradiction, suppose that the $4$-fold Massey product $\langle\chi,\ldots,\chi\rangle$ is defined, then there exists a representation $\rho\colon G \to \bar{\U}_{5}(\F_2)$ such that $\rho_{i,i+1}=\chi$, for $i=1,\ldots,4$. Let $\bar{B}:=\rho(\bar{1})\in \bar{\U}_{5}(\F_2)$. Then all entries of $\bar{B}$ at the positions $(i,i+1)$, $i=1,\ldots, 4$, are equal to 1. Hence $B^2\not=1$, this contradicts with the fact that $\bar{B}$ is the image of an element of order $2$.
\end{ex}
\begin{rmk} Theorem~\ref{thm:p>2} is a generalization of \cite[Proposition 4.5]{MT1}. In \cite{MT1} we use the latter result to provide   an explanation to a part of the well-known Artin-Schreier's theorem \cite{AS1,AS2}  (respectively, Becker's theorem \cite{Be}) which says that the absolute Galois group $G_F$ (respectively, its maximal pro-$p$-quotient $G_F(p)$) of any field $F$ cannot have an element of odd prime order. 

We now use Theorem~\ref{thm:p>2} to give an explanation to the full   Artin-Schreier theorem (respectively, Becker's theorem) which  further says that $G_F$ (respectively, $G_F(2)$) cannot have elements of finite order greater than 2. In fact, it suffices to show that the absolute Galois group $G_F={\rm Gal}(F_{\rm sep}/F)$ cannot be of order 4. First note that Theorem~\ref{thm:p>2} and Example~\ref{ex:p-cyclic} imply that for any field $L$ if $G_L\simeq \Z/2\Z$ then $-1\not\in L^2$. (Not suprisingly $L^2$ means the set of all squares in $L$.)

Now suppose that  $|G_F|=4$.  Let $H$ be a subgroup of $G_F$ of order 2 and let $K$ be the fixed field $(F_{\rm sep})^H$. Then $G_K \simeq \Z/2\Z$.  Thus $-1$ is not in $K^2$, in particular $-1$ is not in $F^2$. Let $ L:= F(\sqrt{-1})$. Then ${\rm Gal}(L/F)$ is of order 2. Therefore $G_L\simeq \Z/2\Z$, hence  $-1\not\in L^2$, which is impossible.
\end{rmk}

 For $1\leq k\leq s$, we define the group
\[M_{p,k,s}= \big \langle \sigma, \tau\mid \tau^{p^{s+1}}=\sigma^{p^{s+1-k}}=1, \sigma\tau\sigma^{-1}=\tau^{1+p^k}\big \rangle\simeq \Z/p^{s+1}\Z \rtimes \Z/p^{s+1-k}\Z.
\]
We have a natural epimorphism $\lambda\colon M_{p,k,s}\to \Z/p^s\Z$ defined by $\tau\mapsto \bar{1}$, $\sigma\mapsto \bar{0}$. 
Note that when $k=s=1$, $M_{p,1,1}\simeq M_{p^3}$.  The following result is a corollary of the proof of Theorem~\ref{thm:p>2}. This result is also a generalization of \cite[Proposition 10.2]{EM1}.
\begin{cor}
\label{cor:automatic realization}
Let $p>2$  be an odd prime and $s\geq 1$ an integer. Let $G=G_F$ or $G_F(p)$ for a field $F$ containing a primitive $p$-th root of unity. Let $\psi:G\to \Z/p^s\Z$ be an epimorphism. Then 
\begin{enumerate} 
\item If $F$ contains a primitive $p^{s+1}$-th root of unity, then $\psi$ factors through the natural map $\Z/p^{s+1}\to \Z/p^s$.
\item If $F$ does not contain a primitive $p^{s+1}$-th root of unity, then $\psi$ factors through one of the epimorphisms:
\begin{enumerate}
\item the natural map $\Z/p^{s+1}\to \Z/p^s$;
\item the map $\lambda\colon M_{p,k,s}\to \Z/p^s\Z$ defined by $\tau\mapsto \bar{1}$, $\sigma\mapsto \bar{0}$, where $k$ is the largest integer such that $F$ contains a primitive $p^k$-th root of unity.
\end{enumerate}
\end{enumerate}
\end{cor}

\begin{thm}
\label{thm:Massey rigid}
 Let $n\geq 3$ be an integer and let $p$ be an odd prime number. Let $F$ be a $p$-rigid field, which contains a primitive $p$-th root of unity. Then for any $\alpha_1,\ldots,\alpha_n\in H^1(G_F,\F_p)$, the $n$-fold Massey product $\langle \alpha_1,\ldots,\alpha_n\rangle$ contains 0 whenever it is defined.
\end{thm}
\begin{proof} For each $i$, $\alpha_i=\chi_{a_i}$ for some element $a_i\in F^\times$. If one of the $a_i$'s is in $(F^\times)^p$ then the corresponding character $\chi_{a_i}$ is trivial and hence the result follows.

We assume now that all $a_i$'s are not in $(F^\times)^p$. Since $\langle \chi_{a_1},\ldots,\chi_{a_n}\rangle$ is defined, $\chi_{a_1}\cup \chi_{a_2}=0$. This implies that $a_2= a_1^k F^p$ for some $1\leq k<p$ since $F$ is $p$-rigid. Hence $\chi_{a_2}=k\chi_{a_1}$. By the linearity of Massey products (see e.g. \cite[Lemma 6.2.4]{Fe}), it is enough to consider the case $k=1$, i.e., $\chi_{a_2}=\chi_{a_1}$. Similarly, it is enough to consider the case that all $\chi_{a_i}$'s are equal. But in this case  Theorem~\ref{thm:p>2} applies and hence the result follows.
\end{proof}
\newpage
\appendix
\begin{center}
by Ido Efrat, J\'an Min\'a\v{c} and \NDT.
\end{center}
\vspace*{10pt}

In this appendix we construct, for each $n\geq 3$, examples of pro-$p$ groups
which do not satisfy the kernel $n$-unipotent property.
These examples are not realizable as maximal pro-$p$ Galois groups of fields containing a root of unity of order $p$.
They also show that several other related results on the structure of such maximal pro-$p$ Galois groups are not valid for general torsion-free pro-$p$ groups (see the remarks below).

Let us keep the notation being as in Section~\ref{sec:6}.
We fix an integer $N\geq1$.
Let $S$ be a free pro-$p$ group on generators  $x_1,\ldots,x_N$.
Let $R_0$ be the closed normal subgroup of $S$ generated by
\[
r_{ij}=[x_1,x_2][x_i,x_j]^{-1}, \quad 1\leq i<j\leq N,
\]
and let $R$ be any closed normal subgroup of $S$ such that $RS_{(3)}=R_0S_{(3)}$.
Let $G=S/R$ and denote the coset of $x_i$ in $G$ by $\bar x_i$.
We note that $G_{(3)}=RS_{(3)}/R$.
We also fix a set $\sL$ of finite groups, and set
\[
G_\sL=\bigcap_{H\in\sL}\bigcap_\rho\hbox{Ker}(\rho),
\]
where the second intersection is over all group homomorphisms $\rho\colon G\to H$.

\begin{appxprop}
\label{prop A1}
Suppose that $N>|H|$ for every $H\in \sL$.
Then:
\begin{enumerate}
\item[(a)]
$[\bar{x}_1,\bar{x}_2]\not\in G_{(3)}$;
\item[(b)]
$[\bar{x}_1,\bar{x}_2]\in G_\sL$;
\item[(c)]
$G_\sL\not\subseteq G_{(n)}$ for every $n\geq3$.
\end{enumerate}
\end{appxprop}
\begin{proof}
(a) \quad
Let $T$ be the subgroup of $S$ generated by the commutators $[x_i,x_j]$, $1\leq i<j\leq N$.
By \cite[Proposition 1.3.2]{Vo},  the cosets of these commutators in $S_{(2)}/S_{(3)}$ are $\Z/p\Z$-linearly independent.
Therefore we can define a homomorphism $\varphi\colon TS_{(3)}/S_{(3)}\to \Z/p\Z$ by
$\varphi([x_i,x_j]S_{(3)})=1$ for $1\leq i<j\leq N$.
Then $\varphi(r_{ij}\bmod S_{(3)})=0$ for each $1\leq i<j\leq N$,
and hence $RS_{(3)}/S_{(3)}\subseteq \ker(\varphi)$.
On the other hand $[x_1,x_2]\mod S_{(3)}$ is not in $\ker(\varphi)$.
Therefore $[x_1,x_2]\not\in RS_{(3)}$, and hence
$[\bar{x}_1,\bar{x}_2]\not\in G_{(3)}$.

(b) \quad
Let $H\in \sL$ and let $\rho\colon G\to H$ be a group homomorphism.
Since $N>|H| $, there exist $1\leq i<j\leq N$ such that $\rho(\bar{x}_i)=\rho(\bar{x}_j)$.
Since $[\bar{x}_1,\bar{x}_2]=[\bar{x}_i,\bar{x}_j]$ in $G$, we have
$\rho([\bar{x}_1,\bar{x}_2])=\rho([\bar{x}_i,\bar{x}_j])=[\rho(\bar{x}_i),\rho(\bar{x}_j)]=1$, as desired.

\medskip

(c)  \quad
Use (a), (b), and the inclusion $G_{(n)}\subseteq G_{(3)}$.
\end{proof}

\begin{appxrmks}
\begin{enumerate}
\item[(1)]
Proposition \ref{prop A1}(c) with $\sL=\{\U_n(\F_p)\}$ gives $G_{\langle n\rangle}\not\subseteq G^{(n)}$ for $n\ge3$ and
$N>|\U_n(\F_p)|$.
Thus $G$ does not satisfy the kernel $n$-unipotent condition for any $n$ such that $n\ge3$ and $N>p^{(n-1)n/2}$.
\item[(2)]
In particular, suppose that $N>p^3$.
Then $G$ does not satisfy the kernel $3$-unipotent condition.
By \cite[Theorem D]{EM2} (see also \cite[Example 12.2]{Ef}), this implies that $G$ is not realizable as the maximal pro-$p$
Galois group $G_F(p)$ of any field $F$ containing a root of unity of order $p$.
\item[(3)]
Let $S_i$, $G_i$, $i=1,2,\ldots$, be again the lower central filtrations of $S,G$, respectively.
Take $R=R_0S_3$.
Then $RS_{(3)}=R_0S_{(3)}$.
As $R\leq S_2$ we have $G/G_2\cong S/S_2\cong\Z_p^N$.
Also, $G_2=G_2/G_3=\langle[\bar x_1,\bar x_2]\rangle\cong\Z_p$.
Thus in this case $G$ is an extension of a torsion-free group by a torsion-free group, and therefore is also torsion-free.
 \item[(4)]
Let $p>2$, let $\sL=\{\Z/p^2\Z,M_{p^3}\}$, and suppose that $N>p^3$.
Since $G^{(3)}\subseteq G_{(3)}$, Proposition \ref{prop A1}(c) implies that $G_\sL\not\subseteq G^{(3)}$.
In view of (3), this shows that the property of absolute Galois groups given in \cite[Main Theorem]{EM1} does not hold in general for torsion-free pro-$p$ groups.
Similarly, taking $p=2$ and $\sL=\{\Z/2\Z,\Z/4\Z,D_4\}$, we obtain that the property of absolute Galois groups given in  \cite[Corollary 2.18]{MS2} (see also \cite[Corollary 11.3]{EM1}) does not hold in general for torsion-free pro-$2$ groups.
Finally, taking $p>2$ and $\sL=\{\Z/p\Z,H_{p^3}\}$, we obtain the same for \cite[Theorem D]{EM2}.
\item[(5)]
Proposition \ref{prop A1}(c) with $\sL=\{\U_{p+1}(\F_p)\}$ gives $G_{\langle p+1\rangle}\not\subseteq G^{(3)}$.
Therefore the group $G$ provides us an example of a pro-$p$ group not satisfying the
conclusion in Proposition~\ref{prop:smallest}, part (1).

\end{enumerate}
\end{appxrmks}

\end{document}